\documentclass{amsart}

\usepackage[enumitem,theorem,Rn,hyperref,frak,indicator]{paper_diening}

\usepackage{bbm}

\usepackage{a4wide}
\usepackage{microtype}
\usepackage{pgfplots}
  \usetikzlibrary{intersections, pgfplots.fillbetween}
\usepackage{tikz}
\usetikzlibrary{intersections,calc,arrows.meta,patterns,angles,quotes}
\usetikzlibrary{backgrounds}
\providecommand{\indicator}{\mathbbm{1}}

\providecommand{\Phix}{\Phi(\cdot)}

\definecolor{darkgreen}{rgb}{0,0.6,0}
\definecolor{colorA}{rgb}{0.6,0.0,0.5}

\providecommand{\opac}{0.5}

\title{Lavrentiev gap for some classes of generalized Orlicz functions}
\author{Anna Kh.~Balci}

\author{Mikhail Surnachev }
\address{Anna Kh.~Balci, University Bielefeld, Universit\"atsstrasse 25, 33615
  Bielefeld, Germany.}
\email{akhripun@math.uni-bielefeld.de}

\address{Mikhail Surnachev, Keldysh Institute of Applied Mathematics,  Miusskaya sq. 4, 125047
  Moscow,  Russia.}
\email{peitsche@yandex.ru}

\thanks{Anna Kh.Balci  thanks German Research Foundation (DFG) for support  through the CRC 1283.  The research of Mikhail Surnachev was supported by the Russian Science Foundation under grant 19-71-30004
}



\begin{document}

\begin{abstract} 
 In the present paper we find optimal conditions separating the regular case from the one with Lavrentiev gap for the borderline case of double phase  potencial and related general classes of integrands. We present new results on density of smooth functions.
\end{abstract}
\maketitle

\section{Introduction}

During last decade  the resurgence of interest in different general growth models  has been experienced. Along with by now almost classical variable~$p(x)$-integrand, presented in hundred of papers and several books, see ~\cite{DieHHR11,CruFio13,KokMesRafSam16},  different properties of other models were considered. The essential feature of these model is possible the presence of a Lavrentiev gap and related with this lack of regularity, non-density of smooth functions in the corresponding energy space and others.

Positive recent results in this direction are sufficiently many and varied. For example, Colombo and Mingione in~\cite{ColMin15} obtained the regularity results for double-phase potential model~$\Phi(x,t)=\frac 1 p t^p+\frac 1 qa(x)t^q$ if $\frac{q}{p} \leq 1+\frac \alpha{d}$ and $a \in C^{0,\alpha}$. Moreover, bounded minimizers  are automatically~$W^{1,q}$
if~$a \in C^{0,\alpha}$ and~$q \leq p+\alpha$, see the paper~\cite{BarColMin18} by Baroni, Colombo and Mingione.  The sharpness of this results was showed by the authors of this paper and Lars Diening in~\cite{BalDieSur20} by constructing the examples of Lavrentiev gap for this model. The other model is weighed~$p$-energy~$\Phi(x,t)=\frac 1 p a(x) t^p$. If~$a$ itself is a Muckenhoupt weight, then it is well known that
smooth functions are dense, so
$W^{1,\Phix}(\Omega) = H^{1,\Phix}(\Omega)$. For other results on the
density in the context of weighted Sobolev spaces with even variable
exponents~$\Phi(x,t)=a(x)t^{p(x)}$, we refer to~\cite{Sur14,ZhiSur16}. The gradient estimates  for the borderline case of double phase problems with BMO coefficients in nonsmooth domains were obtained by Byun and  Oh in~\cite{ByuOh17}.  Skrypnik and   Voitovych recently proved pointwise continuity of solutions for a general class of elliptic and parabolic equations with nonstandard growth conditions  using the De Giorgi-Ladyzhenskaya-Ural'tseva classes, see ~\cite{skr20}.  More models and the extensive list of the references on the generalized Orlicz functions could be found in the book by Harjulehto and  H\"ast\"o~\cite{HarHas19}.  In the present   paper we study the integrands of the form 
\begin{align}
\label{eq:Phipab}
\Phi_{p,\alpha,\beta}(x,t) \sim \frac{1}{p}t^p \log^{-\beta} (e+t) + a(x) \frac{1}{p}t^p \log^\alpha (e+t)%,\quad -\beta\le \alpha.
\end{align}
and in particular for~$p=2$ and ~\begin{align}
  \label{eq:Phi}                                          
\Phi_{\alpha,\beta}(x,t)= \Phi_{2,\alpha,\beta}(x,t)\sim \frac{1}{2}t^2 \log^{-\beta} (e+t) + a(x) \frac{1}{2}t^2 \log^\alpha (e+t).%, \quad -\beta\le \alpha.                 
\end{align}
where $a$ is a non negative bounded weight. The regularity properties of the integrand of this type for~$\Phi_{0,1}(x,t)$ were studied by Baroni,  Colombo,  Mingione in~\cite{BarColMin15} where it  was called~"the borderline case of double phase potential". In particular they obtained the ~$C^{0,\gamma}_{\loc}$ regularity result for the minimizers provided that the weight~$a(x)$ is~$\log$-H\"older continuous (with some~$\gamma$) and more strong result (any~$\gamma\in (0,1)$) for the case of vanishing~$\log$-Hölder continuous  weight. In comparison with these results we obtain regularity results in the sense of density of smooth functions even if the weight is not continuous. The main result is contained in Theorem~\ref{thm:density}. It gives the full description of the checkerboard-type geometry for the borderline case of  the double-phase potential~$\Phi_{\alpha,\beta}$. More precisely, we give the necessary and sufficient conditions on the parameters~$\alpha,\beta$  for the density of smooth functions.

The crucial point for the study of regularity for these classes of problems is possible Lavrentiev gap. The first example in this direction is the famous Zhikov's checkerboard example  for variable exponent, see~\cite{Zhi86}. This example became the guiding principle for other
models. In 1995 Zhikov~\cite[Example~3.1]{Zhi95} considered the double phase potential, later  generalized by Esposito, Leonetti and Mingione
in~\cite{EspLeoMin04} to the case of higher dimensions and less
regular weights.  The general procedure to construct examples for Lavrentiev gap was presented by the authors of this paper and Lars Diening in~\cite{BalDieSur20}.    

We study the corresponding energy, which given by the integral functional 

\begin{align*}
\mathcal{F}(u)=\int_\Omega \Phi_{p,\alpha,\beta}(x,t)\, dx
\end{align*}
and closely related functionals 
\begin{equation}\label{funct_2}
\mathcal{G}(u) = \mathcal{F}(u) + \int_\Omega b\cdot \nabla u\, dx
\end{equation}

We provide examples of the Lavrentiev gap for~$\Phi_{\alpha,\beta}(x,t)$ using one-saddle point construction, which is similar to the initial one from Zhikov's checker-board examples. The energy~$\mathcal{F}$ defines a
generalized Sobolev-Orlicz space~$W^{1,\Phix}(\Omega)$ and its
counterpart~$W^{1,\Phix}_0(\Omega)$ with zero boundary values, see
Section~\ref{ssec:gener-orlicz-space} for the precise definition of
the spaces. Then the above Lavrentiev gap can be also written as 
\begin{align}
\label{eq:E12}
 \mathcal{E}_1:= \inf \mathcal{G}\big(W^{1,\Phix}_0(\Omega) \big)
  &<
    \inf \mathcal{G}\big(H^{1,\Phix}_0(\Omega)\big):= \mathcal{E}_2,
\end{align}
where $H^{1,\Phix}_0(\Omega)$ is the closure of~$C^\infty_0(\Omega)$
 functions in~$W^{1,\Phix}(\Omega)$.

In the second part of the paper we study more general integrands of  double phase potential type
\begin{equation}\label{Phi_def1}
 \Phi(x,t):=\phi(t)+a(x)\psi(t).
\end{equation}
The main result is formulated in the Theorems~\ref{thm3}, \ref{thm:density1}.

\section{Energy and Generalized Orlicz Spaces}
\label{ssec:gener-orlicz-space}

In this section we introduce the necessary function spaces, the so
called generalized Orlicz and Orlicz-Sobolev spaces.

We assume that~$\Omega \subset \Rd$ is a Lipschitz domain of finite
measure.  Later in our
applications we will only use~$\Omega= (-1,1)^2$.

We say that~$\oldphi\,:\, [0,\infty) \to [0,\infty]$ is an Orlicz
function if $\oldphi$ is convex, left-continuous, $\oldphi(0)=0$,
$\lim_{t \to 0} \oldphi(t)=0$ and
$\lim_{t \to \infty} \oldphi(t)=\infty$.  The conjugate Orlicz
function~$\oldphi^*$ is defined by
\begin{align*}
  \oldphi^*(s) &:= \sup_{t \geq 0} \big( st - \oldphi(t)\big).
\end{align*}
In particular, $st \leq \oldphi(t) + \oldphi^*(s)$.

In the following we assume that
$\Phi\,:\, \Omega \times [0,\infty) \to [0,\infty]$ is a generalized
Orlicz function, i.e. $\Phi(x, \cdot)$ is an Orlicz function for
every~$x \in \Omega$ and $\Phi(\cdot,t)$ is measurable for
every~$t\geq 0$. We define the conjugate function~$\Phi^*$ point-wise,
i.e.  $\Phi^*(x,\cdot) := (\Phi(x,\cdot))^*$.

We further assume the following additional properties:
\begin{enumerate}
\item We assume that~$\Phi$ satisfies the $\Delta_2$-condition,
  i.e. there exists~$c \geq 2$ such that for
  all~$x \in \Omega$ and all~$t\geq 0$
  \begin{align}
    \label{eq:phi-Delta2}
    \Phi(x,2t) &\leq c\, \Phi(x,t). 
  \end{align}
\item We assume that~$\Phi$ satisfies the~$\nabla_2$-condition,
  i.e. $\Phi^*$ satisfies the~$\Delta_2$-condition. As a consequence,
  there exist~$s>1$ and~$c> 0$ such that for all $x\in \Omega$,
  $t\geq 0$ and $\gamma \in [0,1]$ there holds
  \begin{align}
    \label{eq:phi-Nabla2}
    \Phi(x,\gamma t) \leq c\,\gamma^s \,\Phi(x,t).
  \end{align}
\item We assume that $\Phi$ and~$\Phi^*$ are proper, i.e. for
  every~$t\geq 0$ there holds $\int_\Omega \Phi(x,t)\,dx< \infty$ and
  $\int_\Omega \Phi^*(x,t)\,dx < \infty$.
\end{enumerate}

We assume that
\begin{align*}
 -c_0+c_1\abs{t}^{p_-}\le \Phi(x,t)\le c_2\abs{t}^{p_+}+c_0,
\end{align*}
where~$1<p_-\le p_+<\infty$,~$c_0\ge 0$,~$c_1,c_2>0$.

Let $L^0(\Omega)$ denote the set of measurable function on~$\Omega$
and $L^1_{\loc}(\Omega)$ denote the space of locally integrable 
functions. 
We define the generalized Orlicz norm by
\begin{align*}
  \norm{f}_{L^{\Phix}(\Omega)} &:= \inf \biggset{\gamma > 0\,:\, \int_\Omega
                    \Phi(x,\abs{f(x)/\gamma})\,dx \leq 1}.
\end{align*}

Then generalized Orlicz space~$L^{\Phix}(\Omega)$ is defined as the
set of all measurable functions with finite generalized Orlicz norm
\begin{align*}
  L^{\Phix}(\Omega) &:= \bigset{f \in L^0(\Omega)\,:\, \norm{f}_{L^{{\Phix}}(\Omega)}<\infty
}.
\end{align*}

For example the generalized Orlicz function $\Phi(x,t) = t^p$
generates the usual Lebesgue space~$L^p(\Omega)$.

The $\Delta_2$-condition of~$\Phi$ and~$\Phi^*$ ensures that our space
is uniformly convex. The condition that~$\Phi$ and $\Phi^*$ are proper
ensure that $L^{\Phix}(\Omega) \embedding L^1(\Omega)$ and
$L^{\Phix}(\Omega) \embedding L^1(\Omega)$. Thus $L^{\Phix}(\Omega)$
and $L^{\Phix}(\Omega)$ are Banach spaces.

We define the generalized Orlicz-Sobolev space~$W^{1,\Phix}$ as
\begin{align*}
  W^{1,\Phix}(\Omega) &:= \set{w \in W^{1,1}(\Omega)\,:\,
                        \nabla w \in L^{\Phix}(\Omega)},
\end{align*}
with the norm
\begin{align*}
 \norm{w}_{W^{1, \Phi(\cdot)}(\Omega)}:=\norm{w}_{L^1(\Omega)}+\norm{\nabla w}_{ L^{\Phi(\cdot)}(\Omega)}.
\end{align*}

In general smooth functions are not dense
in~$W^{1,\Phix}(\Omega)$. We define~$H^{1,\Phix}(\Omega)$ as  
\begin{align*}
  H^{1,\Phix}(\Omega) &:= \big(\text{closure of~$C^\infty(\Omega) \cap W^{1,\Phix}(\Omega)$
                        in~$W^{1,\Phix}(\Omega)$}\big).
\end{align*}
See \cite{DieHHR11} and~\cite{HarHas19} for further properties of
these spaces.

We also introduce the corresponding spaces with zero boundary values
as
\begin{align*}
  W^{1,\Phix}_0(\Omega) &:= \set{w \in W^{1,1}_0(\Omega)\,:\,
                        \nabla w \in L^{\Phix}(\Omega)}
\end{align*}
with the  norm 
\begin{align*}
 \norm{w}_{W_0^{1, \Phi(\cdot)}(\Omega)}:=\norm{\nabla w}_{ L^{\Phi(\cdot)}(\Omega)},
\end{align*}
and   
\begin{align*}
  H_0^{1,\Phix}(\Omega) &:= \big(\text{closure of~$C^\infty_0(\Omega)$
                        in~$W_0^{1,\Phix}(\Omega)$}\big). 
\end{align*}
The space~$W_0^{1,\Phix}(\Omega)$ is exactly those function, which can
be extended by zero to~$W^{1,\Phix}(\Rd)$ functions.

Let us define our
energy~$\mathcal{F}\,:\, W^{1,\Phix}(\Omega) \to \setR$ by
\begin{align*}
  \mathcal{F}(w) &:= \int_\Omega \Phi(x,\abs{\nabla w(x)})\,dx.
\end{align*}
In the language of function spaces~$\mathcal{F}$ is a semi-modular
on~$W^{1,\Phix}(\Omega)$ and a modular on~$W^{1,\Phix}_0(\Omega)$.

\begin{definition}
\label{def:regular}
 An integrand~$\Phi(x,t)$ is said to be regular in the domain~$\Omega$ if for all~$u\in W^{1,\Phi(\cdot)}_0(\Omega)$ with ~$\mathcal{F}(u)<\infty$ there exists a smooth sequence~$u_\epsilon\in C_0^\infty(\Omega)$ such that
 \begin{enumerate}
  \item $u_\epsilon \to u$  in~$W_0^{1,1}(\Omega)$;
  \item $\lim_{\epsilon\to 0} \int_\Omega \Phi(x,\nabla u_\epsilon)\, dx=\int_\Omega \Phi(x,\nabla u)\, dx$.
 \end{enumerate}

\end{definition}
Direct from the Definition~\ref{def:regular} follows that if the integrand~$\Phi(x,t)$ is regular, then~$\mathcal{E}_1=\mathcal{E}_2 $, so there is no Lavrentiev gap. This is equivalent to $H_0^{1,\Phi(\cdot)}(\Omega)=W_0^{1,\Phi(\cdot)}(\Omega)$. Indeed, the reverse implication (from $H=W$ to regularity) is obvious. On the other hand, from Definition~\ref{def:regular} by Scheffe's theorem it follows that $\Phi(x,\abs{\nabla u_\epsilon})$ converges to $\Phi(x,\abs{\nabla u})$ in $L^1(\Omega)$. From this, convexity of $\Phi(x,\cdot)$ and the $\triangle_2$ condition it follows that $\Phi(x, \abs{\nabla u_\epsilon -\nabla u}) \lesssim \Phi (x,\abs{\nabla u_\epsilon}) +  \Phi(x,\abs{\nabla u})$ is equiintegrable, therefore it converges to zero. Thus $u_\epsilon$ is a sequence of $C_0^\infty(\Omega)$ functions approximating $u$ in $W_0^{1,\Phi(\cdot)}(\Omega)$, and so $H_0^{1,\Phi(\cdot)}(\Omega)=W_0^{1,\Phi(\cdot)}(\Omega)$.

We use the following lemma due to Zhikov. 

\begin{lemma}[\cite{Zhi95}, Lemma 2.1]
\label{lemma:zhikov}
Let~$\Omega$ be the star-shaped with respect to the origin domain. Assume, that there exist functions~$\Phi_\epsilon(x,t)$ such that~$\Phi_\epsilon$ are convex with respect to~$t$ and measurable with respect to~$x$ and let the following relations hold. 
\begin{enumerate}
 \item\label{eq:1} $\Phi_\epsilon(x,0)=0$; 
 \item\label{eq:2} $c_1 \Phi(x,t)\le \Phi_\epsilon(x,t)+c_3$ for~$x\in \bar{\Omega}$,~$t\le M\epsilon^{- \frac d {p_-} }$;
 \item \label{eq:3} $\Phi_\epsilon(x,t)\le c_2\Phi(x,t)+c_3$ for~$\abs{x-y}\le 2k\epsilon$,~$t\in\setR_+$,
\end{enumerate}
with~$k,M,c_1,c_2,c_3>0$ and~$0<\epsilon< \epsilon_0$. Then the integrand~$\Phi(x,t)$ is regular.
 \end{lemma}
 The function~$\omega(x)$ is the modulus of continuity for the weight~$a(x)$ if
 \begin{align*}
  \abs{a(x)-a(y)}\le \omega(\abs{x-y}) \quad x,y\in\setR^d,\quad \abs{x-y}\le \frac 1 4.
 \end{align*}

By $B_r(x)$ we denote the  ball of radius~$r$ with center at the point~$x$. 
\begin{corollary}
\label{theorem:regularity}
 Let 
 \begin{align*}
  \Phi(x,t):=\phi(t)+a(x)\psi(t),
 \end{align*}
with Orlicz functions~$\phi(t)$, $\psi(t)$.
 Assume that the  weight~$a(x)$ is non-negative, bounded and  has the modulus of continuity
 %~$\log$-Hölder continuous:
  \begin{align}\label{ac_z1}
  \omega(\epsilon)\le k_0\frac{\phi(t)}{\psi(t)},\quad 1\le t\le \epsilon^{-d}.
\end{align}
Then the integrand~$\Phi_{\alpha,\beta}$ is regular.  In particular, $C_0^\infty(\Omega)$ is dense in~$W_0^{1,\Phi(\cdot)}(\Omega)$.
\end{corollary}
\begin{proof}
 We set 
 \begin{align*}
  \Phi_\epsilon(x,t):= \phi(t) + a_\epsilon(x) \psi(t)
 \end{align*}
with  
\begin{align*}
 a_\epsilon(x)=\min\set{a(y),\  y\in \bar{\Omega}\cap B_{\varepsilon} (x)} . 
\end{align*}
From the definition of~$a_\epsilon$ it follows that 
\begin{align*}
 a_\epsilon(x)\le a(y) \quad \text{ if }\abs{x-y}<\epsilon,
\end{align*}
and thus condition \ref{eq:3} of Lemma~\ref{lemma:zhikov} is verified.

Now, using~\eqref{ac_z1} we get 

\begin{align*}
\Phi(x,t) &\leq \phi(t) + a_\epsilon(x)\psi(t) + \omega(\epsilon)\psi(t) \\
&\leq \Phi_\epsilon(x,t) + k_0\phi(t)  \leq (k_0+1)  \Phi_\epsilon(x,t)
\end{align*}
for $1\le t \leq \varepsilon^{-d}$. Since~$\Phi(x,t)\le k_1$ for~$t\le 1$,  we see that the  condition~\ref{eq:2} of Lemma~\ref{lemma:zhikov} is fulfilled. It is easy to see that condition~\ref{eq:1} also holds. 
\end{proof}

 \begin{remark}
  Let us mention, that Corollary~\ref{theorem:regularity} holds  for any~$d$ independently on the dimension. Note, that the  further results  are for~$p=d=2$.
 \end{remark}
In particular we see that for~$\Phi_{p,\alpha,\beta}$ defined in~\eqref{eq:Phipab} all conditions of Corollary~\ref{theorem:regularity} are fulfilled provided that the weight~$a(x)$ has the modulus of continuity
 \begin{align*}
  \omega(r)\le \frac{k_0}{\log^{\alpha+\beta}(r^{-1})},\quad \text{if }r\le \frac 1 4.
 \end{align*}

From the point of view of functional spaces we are interested mostly in the Zygmund classes
\begin{align*}
L^p\log^\gamma L,\quad 1\le p<\infty, \quad \gamma\in \setR,
 \end{align*}
which were studied, for example by Iwaniec and Sbordone in  \cite[Section 18]{IwaSbo00}. These classes correspond to~$\phi(t)= t^p\log^\alpha (e+t)$. These Orlicz classes are convex if~$\gamma\ge 1-p$. This issue is not important for us, since we are only interested in the behaviour of the integrand for large values of~$t$, hence we can always replace it by an appropriate convex Orlicz substitution. 
The norm, defined as
\begin{align*}
\norm{f}_{L^p\log^\alpha L}:= \bigg(\int_\Omega \abs{f}^p\log^\alpha \big(e+\frac{\abs{f}}{\norm{f}_p} \big)\, dx \bigg)^{\frac 1 p} 
\end{align*}
is equivalent to the Luxemburg norm. 
The following H\"older-type inequality is valid 
\begin{align*}
 \norm{AB}_{L^c\log^\gamma L}\le C_{\alpha,\beta}(a,b) \norm{A}_{L^a \log^\alpha L} \norm{B}_{L^b \log^\beta L}
\end{align*}
for 
$$
\frac{1}{c} = \frac{1}{a} + \frac{1}{b}, \quad \frac{\gamma }{c} = \frac{\alpha}{a} + \frac{\beta}{b}.
$$
And for
\begin{align}
\label{eq:duality}
\phi(t)=\frac 1p \abs{t}^p \log^\gamma (e+\abs{t}),\quad \phi^*(t)\approx \frac{1}{p'} |t|^{p'}\log^{\frac \gamma {1-p}} (e+\abs{t}).
\end{align}
Also 
$$
\left(a(x) \abs{t}^p \log^\gamma (e+\abs{t}) \right)^* \approx a(x)^{1-p' } |t|^{p'} \log^{\frac \gamma {1-p}} \left(e+\frac{\abs{t}}{a(x)} \right).
$$

\section{Borderline case of double phase potential }
\label{sec:Checkerboard setup}
In this section we consider the integrand~$\Phi_{\alpha,\beta}$ defined in~\eqref{eq:Phi} in planar domains. We describe its regularity for  the checkerboard geometry. 
Further in this section we drop $\alpha,\beta$ from the notation $\Phi_{\alpha,\beta}$ and write 
 $$
\Phi(x,t) = \varphi(t) + a(x) \psi(t), \quad \varphi(t)= t^2\log^{-\beta}(e+t), \quad \psi(t) = t^2 \log^\alpha(e+t).
 $$
 So instead of $W^{1, \Phi_{\alpha,\beta}(\cdot)}(\Omega)$ we write simply $W^{1,\Phi(\cdot)}(\Omega)$, and the same convention  is used for other spaces.

We denote ~$\Omega= (-1,1)^2$ and  use the notation from \cite{BalDieSur20}.

\begin{definition}[Checkerboard setup]
  \label{def:zhikov}
  Let
  \begin{align*}
 x=(x_1,x_2),  \quad x_1, x_2\in\setR.
  \end{align*}
We define $u_2$, $A_2$ and $b_2$ on~$\setR^2$ by
  \begin{align*}
    u_2 &:= \frac 12 \sgn(x_2)\, \theta\bigg( \frac{\abs{x_2}}{\abs{x_1}} \bigg), 
    \\
    A_2 &:= \theta\bigg( \frac{\abs{x_1}}{\abs{x_2}} \bigg)
            \frac{1}{\sigma_{1}} \abs{x_1}^{-1}
              \begin{pmatrix}
                0 & -x_1\\
                x_1 & 0
              \end{pmatrix}, 
    \\
    b_2 &:= \divergence A_2,
  \end{align*}
  where $\sigma_{1}=2$ is the ``surface area'' of the $1$-dimensional
  sphere and ~$\theta \in C^\infty_0((0,\infty))$ is such that
  $\indicator_{(\frac 12,\infty)} \leq \theta \leq \indicator_{(\frac
    14, \infty)}$,  $\norm{\theta'}_\infty \le 6$.
\end{definition}

The matrix divergence is taken rowwise, i.e for matrix~$A=\set{A_{ij}}$ we define~$(\divergence A)_i=\partial_j A_{ij}$. That is,
$$
b_2 =  \nabla^\perp v, \quad v= \frac{1}{2} \sgn(x_1)\, \theta\bigg( \frac{\abs{x_1}}{\abs{x_2}} \bigg) .  
$$

The following properties of functions~$u_2$,~$b_2$,~$A_2$ were proved in~\cite{BalDieSur20}. 

\begin{proposition}
  \label{pro:prop-uAb-d}
  There holds
  \begin{enumerate}
  \item
    $u_2 \in L^\infty(\setR^2) \cap W^{1,1}_{\loc}(\setR^2) \cap C^\infty(\setR^2
    \setminus \set{0})$,
  \item $A_2 \in W^{1,1}_{\loc}(\setR^2) \cap C^\infty(\setR^2 \setminus \set{0})$,
  \item
    $b_2 \in L^1_{\loc}(\setR^2) \cap C^\infty(\setR^2 \setminus \set{0})$.
  \item The following estimates hold
  \begin{align*}
    \begin{alignedat}{2}
      \abs{\nabla u_2} &\lesssim \abs{x_2}^{-1} \indicator_{ \set{ 2
          \abs{x_2} \leq \abs{x_1} \leq 4 \abs{x_2}}} &&\eqsim
      \abs{x_1}^{-1} \indicator_{ \set{ 2 \abs{x_2} \leq
          \abs{x_1} \leq 4 \abs{x_2}}}
      \\
      \abs{b_2} &\lesssim \abs{x_2}^{-1}
      \indicator_{ \set{ 2 \abs{x_1} \leq \abs{x_2} \leq 4
          \abs{x_1}}}&&\eqsim \abs{x_1}^{-1}
      \indicator_{ \set{ 2 \abs{x_1} \leq \abs{x_2} \leq 4
          \abs{x_1}}} .
     \end{alignedat}
  \end{align*}
  \item  $\abs{\nabla u_2}\cdot \abs{b_2} = 0$.
 \vskip 12pt
 
 \item   $
    \int_{\partial \Omega} (b_2 \cdot \nu) u_2\,dS = 1.$
\end{enumerate}

  \end{proposition}

We denote~
\begin{align*}
C_+&=\set{x: \abs{x_1}<x_2 }\cap \Omega,\\
C_-&=\set{x:\abs{x_1}< -x_2}\cap \Omega .
\end{align*}
 The weight~$a(x)$ as defined as
 \begin{align*}
  a(x)=\begin{cases}
        1,\text{ if } \abs{x_1}<\abs{x_2}\\
        0,\text{ if }\abs{x_1}\ge\abs{x_2}.
       \end{cases}
\end{align*}
 
 The Figure 1  shows the function~$u$, the $(2,1)$-component of~$A$ and a
possible weight~$a(x)$. It jumps from~$0$ to $1$ and in the filled regions is smooth transition.  
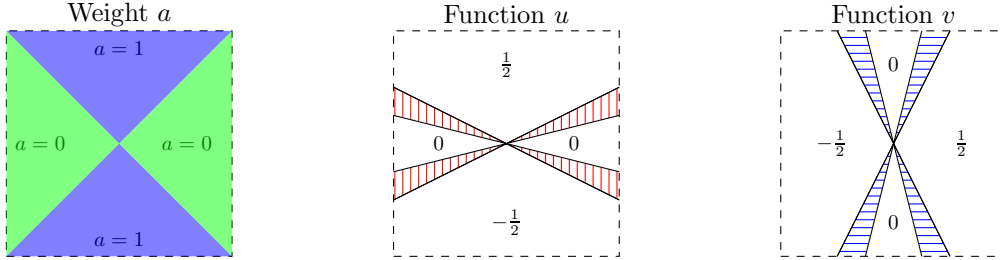
\begin{figure}[!ht]
\label{fig:revised-checkerboard}
    \begin{minipage}[t]{0.31\textwidth}
      \begin{tikzpicture}[scale=1.5]
        \node at (0,1.15) {Weight~$a$};

        \draw[dashed] (-1,-1) -- (-1,+1) -- (+1,+1) -- (+1,-1) --cycle;
        \node at (0,.85) {\scalebox{0.8}{$a=1$}};
        \node at (0,-.85) {\scalebox{0.8}{$a=1$}};
        \node at (.6,0) {\scalebox{0.8}{$a=0$}};
        \node at (-.7,0) {\scalebox{0.8}{$a=0$}};

        \fill[blue, opacity=\opac] (0,0) -- (1,1)--(-1,1)-- cycle;

        \fill[blue, opacity=\opac] (0,0) -- (1,-1)--(-1,-1)-- cycle;

        \fill[green, opacity=\opac] (0,0) -- (1,1)--(1,-1)-- cycle;
        
        \fill[green, opacity=\opac] (0,0) -- (-1,1)--(-1,-1)-- cycle;
      \end{tikzpicture}
    \end{minipage}
    \hfill
    \begin{minipage}[t]{0.31\textwidth}
      \begin{tikzpicture}[scale=1.5]
        \node at (0,1.15) {Function~$u$};
        \draw[dashed] (-1,-1) -- (-1,+1) -- (+1,+1) -- (+1,-1) --cycle;
        \node at (0,.7) {\scalebox{0.8}{$\frac 12$}};
        \node at (0,-.7) {\scalebox{0.8}{$-\frac 12$}};
        \draw (-1,-1/2) -- (1,1/2);
        \draw (-1,1/2)-- (1,-1/2);
        \node at (+.6,0) {\scalebox{0.8}{$0$}};
        \node at (-.6,0) {\scalebox{0.8}{$0$}};
        \filldraw[pattern=vertical lines, pattern color=red] (-1,-1/2) -- (0,0) -- (-1,-1/4);
        \filldraw[pattern=vertical lines, pattern color=red] (-1,+1/2) -- (0,0) -- (-1,+1/4);
        \filldraw[pattern=vertical lines, pattern color=red] (+1,-1/2) -- (0,0) -- (+1,-1/4);
        \filldraw[pattern=vertical lines, pattern color=red] (+1,+1/2) -- (0,0) -- (+1,+1/4);
      \end{tikzpicture}
    \end{minipage}
    \hfill
    \begin{minipage}[t]{0.31\textwidth}
      \begin{tikzpicture}[scale=1.5]
        \node at (0,1.15) {Function~$v$};
        \draw[dashed] (-1,-1) -- (-1,+1) -- (+1,+1) -- (+1,-1) --cycle;
        
        \node at (.6,0) {\scalebox{0.8}{$\frac 12$}};
        \node at (-.55,0) {\scalebox{0.8}{$-\frac 12$}}; 
        \draw (-1/2,-1)--(+1/2,1);
        
        \draw(+1/2,-1)-- (-1/2,+1) ;
        \node at (0,.7) {\scalebox{0.8}{$0$}};
        \node at (0,-.7) {\scalebox{0.8}{$0$}};
        
        \filldraw[pattern=horizontal lines, pattern color=blue] (-1/2,-1) -- (0,0) -- (-1/4,-1);
        \filldraw[pattern=horizontal lines, pattern color=blue] (+1/2,-1) -- (0,0) -- (+1/4,-1);
        \filldraw[pattern=horizontal lines, pattern color=blue] (-1/2,+1) -- (0,0) -- (-1/4,+1);
        \filldraw[pattern=horizontal lines, pattern color=blue]
        (+1/2,+1) -- (0,0) -- (+1/4,+1);
      \end{tikzpicture} 
\end{minipage}%
    \caption{One saddle point}
\end{figure}

 \begin{theorem}
\label{thm:density}
For the integrand~$\Phi=\Phi_{\alpha,\beta}$ the equality $H_0^{1,\Phi(\cdot)}=W_0^{1,\Phi(\cdot)}$   is valid when~$\min(\alpha,\beta)\leq 1$.% $\alpha\le 1$ for any~$\beta$ or $\alpha>1$ and~$\beta\le 1$.  

If ~$\alpha>1$ and~$\beta>1$, then  $H_0^{1,\Phi(\cdot)}\neq W_0^{1,\Phi(\cdot)}$ and there exists $b \in L^{\Phi^*(\cdot)}(\Omega)$ such that there is Lavrentiev gap for $\mathcal{G}(\cdot)$ defined by \eqref{funct_2}:
 $$
 \inf \mathcal{G} (W_0^{1,\Phi(\cdot)}) <  \inf \mathcal{G} (H_0^{1,\Phi(\cdot)}). 
 $$
 Moreover, in this case the codimension of $H_0^{1,\Phi(\cdot)}$ in $W_0^{1,\Phi(\cdot)}$ is one.

 \end{theorem}

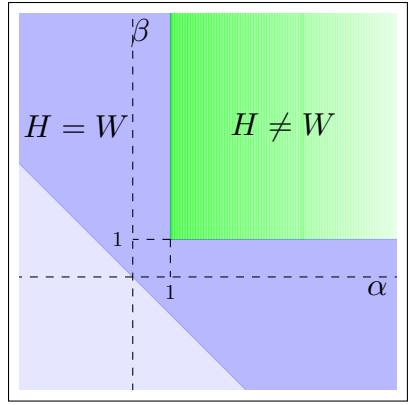
\begin{figure}[!ht]
\label{fig:main}
      \begin{tikzpicture}[framed,scale=0.5]
\clip (-3,-3) rectangle (7,7);
      % \node at (0,1.15) {Weight~$a$};

%         \draw[dashed] (-1,-1) -- (-1,+1) -- (+1,+1) -- (+1,-1) --cycle;
        \draw[left color=green, right color=white, opacity=\opac,draw=none](1,1)--(8,1)--(8,8)--(1,8);
%         %\filldraw[blue, opacity=\opac](0,0)--(0,1)--(-1,1)--(-1,-1)--(1,-1)--(1,0)--(0,0);
%          \filldraw[blue, opacity=\opac](-3,1)--(-1,-1)--(1,-3)--(1,0)--(0,0)--(0,1)--(-3,1);
clop
%\draw[name path=A] plot (\x,-\x);
\draw[name path=B, blue, opacity=0.2] (9,-7)--(9,1)--(1,1)--(1,8)--(-6,8);
\draw[name path=A, blue, opacity=0.2] (-7,8)--(-8,8)--(8,-8)--(9,-8);
\draw[name path=C, blue, opacity=0.2] (-9,7)--(-9,-7)--(9,-7);
\tikzfillbetween[of=B and A]{blue, opacity=0.2};
\tikzfillbetween[of=B and C]{blue, opacity=0.1};

\draw[dashed](1,1)--(0,1);
\draw[dashed](1,1)--(1,0);
\node at (1,-0.4) {\scalebox{0.8}{$1$}}; 
\node at (-0.4,1) {\scalebox{0.8}{$1$}}; 
\draw[->,dashed] (0,-6)--(0,8);
       \draw[->,dashed] (-4,0)--(9,0);
    \node at (0.2,6.5) {\scalebox{1.2}{$\beta$}};
      \node at (6.5,-0.3) {\scalebox{1.2}{$\alpha$}}; 
      \node at (4,4) {\scalebox{1.2}{$H\neq W$}}; 
      \node at (-1.5,4) {\scalebox{1.2}{$H= W$}}; 
      \end{tikzpicture}
 
 \caption{Illustration for Theorem~\ref{thm:density}}
 \end{figure}

The proof of this theorem  is split into several lemmata.

\begin{lemma}
\label{lem:Lavr}
If $\alpha>1$ and $\beta>1$ then $u_2 \in W^{1,\Phi(\cdot)}(\Omega)$ and $b_2 \in L^{\Phi^* (\cdot)}(\Omega)$.
\end{lemma}
\begin{proof}
Clearly
\begin{align*}
\{\abs{\nabla u_2} \neq 0\} \subset \{a=0\},& \quad  \{|\abs{b_2} \neq 0\} \subset \{a=1\}.
 %\abs{\nabla u_2} \lesssim \abs{x_2}^{-1} \indicator_{ \set{ 2
  %        \abs{x_2} \leq \abs{x_1} \leq 4 \abs{x_2}}} &\eqsim
   %   \abs{x_1}^{-1} \indicator_{ \set{ 2 \abs{x_2} \leq
   %       \abs{x_1} \leq 4 \abs{x_2}}},\\
  %         \abs{b_2} &\eqsim \abs{x_2}^{-1}
  %     \indicator_{ \set{ 2 \abs{x_1} \leq \abs{x_2} \leq 4 \abs{x_1}}} .          
\end{align*}
%see Propositions~\ref{pro:prop-uAb-d}%-\ref{lem:bd-int}.

Then by Proposition~\ref{pro:prop-uAb-d} we get
    \begin{gather*}
   \int_\Omega \Phi^*(x, \abs{b_2})\, dx \lesssim \int_{\Omega \cap \{a=1\}} |b_2|^2 \log^{-\alpha}(e + |b_2|)\, dx \\
   \lesssim  \int_0^{2} \frac{dt}{t \log^\alpha (e+t)}<\infty
    \end{gather*}
provided~$\alpha>1$. And
    \begin{gather*}
   \int_\Omega \Phi(x, \abs{\nabla u_2})\, dx \lesssim \int_{\Omega \cap \{a=0\}} |\nabla u_2|^2 \log^{-\beta}(e + |\nabla u_2|)\, dx \\
\lesssim  \int_0^{2} \frac{dt}{t \log^\beta (e+t)}<\infty,
\end{gather*}
provided~$\beta>1$.
\end{proof}
Let~$D_h=C_+\cap \set{0< x_2< h}$. 
\begin{lemma}
\label{lem:embedding}
If $\alpha>1$ and $ u \in W^{1,\Phi(\cdot)}(\Omega)$ then it is continuous in $\overline{C_\pm}$ with modulus of continuity~$$\omega(t)\leq C(\alpha) \|\nabla u\|_{L^{\psi(\cdot)} (C_\pm)}  \log^{\frac{1-\alpha}{2}}(1/t),\quad t<1/e.$$ Moreover, 
$$\omega(t)   \log^{\frac{\alpha-1}{2}}(1/t) \to 0 \text{ as }t\to 0.$$  %, where the norm of $\nabla u$ is in $L^{\psi(\cdot)} (C_\pm)$.
\end{lemma}

\begin{proof}
We start with the well-known estimate of~$u$ in terms of Riesz potential  and H\"older inequality
 \begin{align*}
  \abs{u(x)-\mean{u}_{D_h}}&\lesssim  \int_{D_h}\frac{\abs{\nabla u(y)}}{\abs{x-y}}\, dy \lesssim \norm{\nabla u}_{L^{\psi(\cdot)}(D_h)}\norm{\abs{x-\cdot}^{-1}}_{L^{\psi^*(\cdot)}(D_h)},\quad x\in \bar D_h.
 \end{align*}
 By definition
\begin{align*}
  \norm{\abs{x-\cdot}^{-1}}_{L^{\psi^*(\cdot)}(D_h)}&\leq \norm{r^{-1}}_{L^{\psi^*(\cdot)}(B_{2h}(0))}\\
  &\lesssim \inf \left\{ \lambda>0\,: \, 2\pi\int_{0}^{2h}\psi^*((\lambda r)^{-1}) r\, dr\leq 1 \right\}  \\
  &\lesssim  \inf \left\{ \lambda>0\,: \, \int_{0}^{2h}(\lambda r)^{-2} \log^{-\alpha} (e+(\lambda r)^{-1}) r\, dr\leq c(\alpha) \right\}.
 \end{align*}
 Now 
 \begin{gather*}
 \int_{0}^{2h}(\lambda r)^{-2} \log^{-\alpha} (e+(\lambda r)^{-1}) r\, dr = \lambda^{-2}\int_0^{2\lambda h} t^{-1} \log^{-\alpha} (e+ t^{-1})\, dt \\
 = \lambda^{-2} \int_{(2\lambda h)^{-1}}^\infty s^{-1} \log^{-\alpha} (e+s)\, ds = \lambda^{-2} \frac{1}{\alpha-1}\log^{1-\alpha} (e+(2\lambda h)^{-1}).
 \end{gather*}
 And so
\begin{gather*}
  \norm{\abs{x-\cdot}^{-1}}_{L^{\psi^*(\cdot)}(D_h)}   \leq \left(2h\sup\left\{\lambda>0\, :\, \lambda^{2} \log^{1-\alpha} (e + \lambda) \leq c(\alpha) h^{-2} \right\} \right)^{-1}\\
  \leq  c(\alpha)\log^\frac{1-\alpha}{2} \frac{1}{h}.
\end{gather*} 
Then
 \begin{align*}  
  \abs{u(x)-\mean{u}_{D_h}} \leq c(\alpha) \norm{\nabla u}_{L^{\psi(\cdot)} (D_h)} \log^{\frac{1-\alpha}{2}}\frac{1}{h},\quad x\in \bar D_h.
 \end{align*}
 This proves the required continuity at the origin when it is approached from $C_+$ and for other points of $C_{+}$ the proof is by obvious modification. For $C_{-}$ the reasoning is the same.
\end{proof}

 %We consider the integrand~$\Phi_{\alpha,\beta}$ at each quadrant separately: for the horizontal cone with~$a(x)=0$ the corresponding Orlicz function is~$t^2\log^{-\beta}(e+t)$. For the vertical cone and~$a(x)=1$ it is~$t^2\log^\alpha(e+t)$. So away from the origin we have two autonomous Orlicz functions separated by the hyper-plane~$\Sigma$.

Let $\alpha>1$ and define using Lemma~\ref{lem:embedding} the limit values 
\begin{equation}\label{def_upm}
u_{+} = \lim_{C_{+}\ni x \to 0} u(x), \quad u_{-}=\lim_{C_{-}\ni x \to 0} u(x).
\end{equation}

\begin{lemma}
\label{lem:lim_value_coin}
 Assume that~$\alpha>1$, $\beta \le 1$ and $u\in W^{1,\Phi(\cdot)}(\Omega)$. Then $u_{+}=u_{-}$ .
\end{lemma}

\begin{proof}
 Indeed, assume that~$u_+\neq u_{-}$.  We assume without loss that $|u_+ - u_{-}|=1$. Then for any~$s\in (0,1/4)$ we have
\begin{align*}
\int_{-h}^h \abs{\nabla u(s,x_2)}\, dx_2\ge 1
\end{align*}
and upon integration over $s\in (h/2,h)$, $h\leq 1/4$, this yields
$$
\int_{\Omega_h} \abs{\nabla u}\, dx\ge \frac h 2,
$$
where~$\Omega_h=\set{(x_1,x_2):\frac h 2\le x_1\le h,\quad \abs{x_2}<x_1}$.  Now
$$
 \frac h 2 \lesssim \norm{\nabla u}_{L^{\Phi(\cdot)}(\Omega_h)}\norm{1}_{L^{\Phi^*(\cdot)}(\Omega_h)}
$$
and 
\begin{gather*}
\norm{1}_{L^{\Phi^*(\cdot)}(\Omega_h)} = \norm{1}_{L^{\varphi^*(\cdot)}(\Omega_h)} \\
= \inf \left\{ \lambda>0 \, :\, \int_{\Omega_h} \lambda^{-2} \log^{\beta} (e+\lambda^{-1})\, dx \leq c(\beta) \right\}\\
= \inf \left\{ \lambda>0 \, :\,  \lambda^{-2} \log^{\beta} (e+\lambda^{-1})\, dx \leq c(\beta) h^{-2} \right\} \\
\lesssim h \log^{\beta/2} \frac{1}{h}
\end{gather*}
yield
$$
 \norm{\nabla u}_{L^{\Phi(\cdot)}(\Omega_h)}\ge \epsilon\log^{-\frac \beta 2} \frac 1 h  
$$
with some positive constant $\epsilon = \epsilon(\beta)\in (0,1)$. 

By definition of the Luxemburg norm we have
\begin{align*}
 \int_{\Omega_h}\frac{\abs{\nabla u}^2}{\epsilon^2\log^{-\beta}\frac 1 h}\log^{-\beta}\bigg(e+\frac{|\nabla u|}{\epsilon\log^{-\beta/2}\frac 1 h} \bigg)\, dx \ge 1.
\end{align*}
Therefore for $h\le 2^{-j_0}$ for some~$j_0>0$ there holds
 \begin{align*}
 \int_{\Omega_h} \abs{\nabla u}^2\log^{-\beta}\bigg(e+\frac{|\nabla u |}{\epsilon\log^{-\beta}\frac 1 h} \bigg)\, dx &\ge \epsilon^2\log^{-\beta}\frac 1 h,  \\
 \int_{\Omega_h} \abs{\nabla u}^2\log^{-\beta}\bigg(e+\abs{\nabla u}\bigg)\, dx&\ge \epsilon^2\log^{-\beta}\frac 1 h.
 \end{align*}
Summing the last inequality over $h = 2^{-j}$, $j \geq j_0$, we arrive at 
\begin{align*} 
\int_{\Omega} \abs{\nabla u}^2\log^{-\beta}\bigg(e+\abs{\nabla u}\bigg)\, dx\ge \sum_{j=j_0}^\infty \epsilon^2\log^{-\beta} 2^{j} \ge \frac{\epsilon^2}{\log^{\beta}{2}}\sum_{j=j_0}^\infty \frac 1 {j^\beta} =+\infty
\end{align*}
provided that~$\beta \le 1$. This proves~$u_-=u_+$. \end{proof}

%The proof of the next lemma will be postponed till the next section. 
\begin{lemma}\label{lemma:post}
Let $u \in W^{1,\Phi(\cdot)}(\Omega)$ and $u=0$ in a neighbourhood of the origin. Then $u\in H^{1,\Phi(\cdot)}(\Omega)$. If $u\in W_0^{1,\Phi(\cdot)}(\Omega)$ and $u=0$ in a neighbourhood of the origin then $u\in H_0^{1,\Phi(\cdot)}(\Omega)$.
\end{lemma}
\begin{proof}
By partition of unity, rotation and dilation the proof is reduced to showing the following fact. Let $\tilde{\Phi}(x,t) =\varphi(t) + \tilde a(x) \psi(t)$, where $\tilde a(x) =0$ when $x_2>0$ and $\tilde a(x) = 1$ when $x_2<0$. Denote $Q = \{(x_1,x_2)\, :\, |x_1| + |x_2|<1 \}$. Then $ W^{1,\tilde{\Phi}(\cdot) }(Q) = H^{1,\tilde{\Phi}(\cdot) }(Q)$ and $ W_0^{1,\tilde{\Phi}(\cdot) }(Q) = H_0^{1,\tilde{\Phi}(\cdot) }(Q)$. 

Let $v\in W^{1,\tilde{\Phi}(\cdot)}(Q)$. Denote $Q_+ = Q \cap \{x_2 >0\}$, $Q_{-} = Q \cap \{x_2<0\}$ and  by~$w$ the even extension of the function~$v$ from~$Q_{-}$ to~$Q_{+}$, that is 
$$
w(x_1,x_2) = \begin{cases}
v(x_1,x_2),\quad  & x_2 <0,\\
v(x_1,-x_2), \quad & x_2 \geq 0.
\end{cases}
$$
Set~$z=v-w$.  The function $z$ has zero trace on $\{x_2=0\}$ and vanishes on $Q_{-}$. In the region $Q_{+}$ obviously $\tilde \Phi(x,t) = \varphi(t)$ is independent of $x$. So there exists (by the standard mollification procedure) a sequence $z_\epsilon \in C^\infty(\overline{Q_{+}})$ such that $z_\epsilon =0$ when $\{x_2<\varepsilon\}$ and   $z_\varepsilon \to z$ in $W^{1,\varphi(\cdot)} (Q_+)$, therefore $z_\varepsilon \to z$ in $W^{1,\Phix}(Q)$. On the other hand, $w \in W^{1,\psi(\cdot)}(Q)$ and thus it can be approximated by $w_\epsilon \in C^\infty(\overline{Q})$ in $W^{1,\psi(\cdot)}(Q)$. Take $u_\epsilon = w_\epsilon + z_\epsilon$. Clearly, it converges to $u$ in $W^{1,\Phix}(Q)$. 

For $v\in W_0^{1,\tilde{\Phi}(\cdot)}(Q)$ the proof is the same, but $z\in W_0^{1,\phi(\cdot)}(Q_{+})$, $w \in W_0^{1,\psi(\cdot)}(Q)$, so we can take approximating sequences $z_\epsilon$ from~$C_0^\infty(Q_+)$ and $w_\epsilon$ from $C_0^\infty(Q)$.
\end{proof}

Thus the difference between $W^{1,\Phi(\cdot)}(\Omega)$ and $H^{1,\Phi(\cdot)}(\Omega)$  is in some sense concentrated at the origin (the saddle point). In the next statement we claim that this possible singularity is always removable provided that $\alpha\leq 1$.

\begin{lemma}\label{lemma:remove}
Let $\alpha \leq 1$. Then there exists a sequence of functions $\eta_\epsilon\in C^\infty (\Omega)$, $\epsilon\to 0$, such that $\eta_\epsilon =1$ outside $B_\epsilon(0)$, $\eta_\epsilon =0$ in a neighbourhood of the origin, $0\leq \eta_\epsilon \leq 1$, and $$\int_\Omega \Phi(x,\abs{\nabla \eta_\epsilon})\, dx \to 0 \text{  as }\epsilon\to 0.$$
\end{lemma}

\begin{proof}
Take $0<\epsilon<1/10$ and set 
\begin{align*}
\eta_\epsilon(r)=  \begin{cases}1, \quad &r \geq \epsilon,\\ 
 \frac{\log (1/\epsilon) - \log \log (1/r)}{ \log (1/\epsilon) - \log \log (1/\epsilon)} ,\quad& e^{-1/\epsilon}<r<\epsilon,\\
0, \quad &r\leq e^{-1/\epsilon}.
\end{cases}
\end{align*}
Clearly it is sufficient to show that 
$$
I_\epsilon=\int_\Omega |\nabla \eta_\epsilon|^{2} \log (e + |\nabla \eta_\epsilon|)\, dx \to 0 \quad \text{as} \quad \epsilon \to 0.
$$
We evaluate
\begin{gather*}
I_\epsilon\lesssim \int\limits_{e^{-1/\epsilon}}^{\epsilon}\frac{1}{r^2 \log^2 (1/r) \log^2 (1/\epsilon) } \log \left(\frac{1}{r \log (1/r)  \log (1/\epsilon)} \right)\, r  dr  \\
\leq  \frac{1}{\log^2 (1/\epsilon)}\int\limits_{e^{-1/\epsilon}}^{\epsilon}\frac{1}{r \log (1/r) } \, dr \leq   \frac{1}{ \log (1/\epsilon)}.
\end{gather*}
It remains to send $\epsilon$ to zero. 
\end{proof}

\begin{corollary}\label{corr:le1}
Let $\alpha \leq 1$. If $u\in W^{1,\Phi(\cdot)}(\Omega)$ then $u\in H^{1,\Phi(\cdot)}(\Omega)$. If $u\in W_0^{1,\Phi(\cdot)}(\Omega)$ then $u\in H_0^{1,\Phi(\cdot)}(\Omega)$.
\end{corollary}

\begin{proof}
Any function from $W^{1,\Phi(\cdot)}(\Omega)$ or $W_0^{1,\Phi(\cdot)}(\Omega)$ can be approximated by bounded functions (it is enough to consider standard level cuts). So without loss of generality we  assume that $u\in W^{1,\Phi(\cdot)}(\Omega)$ $\|u\|_{L^\infty(\Omega)} <\infty$ (for $u\in W_0^{1,\Phi(\cdot)}(\Omega)$ the proof is the same). Consider 
$$
u_\epsilon = u \eta_\epsilon,
$$
where $\eta_\epsilon$ is defined in Lemma~\ref{lemma:remove}. Then 
$$
\nabla u - \nabla u_\epsilon = -u \nabla \eta_\epsilon + (1-\eta_\epsilon) \nabla u
$$
and by the $\triangle_2$ property
\begin{align*}
\int_{ \Omega} \Phi(x,\abs{\nabla (u-u_\epsilon)} )\, dx  \lesssim \int_\Omega \Phi(x,|\nabla \eta_\epsilon|) \,dx+\int_\Omega \Phi(x,(1-\eta_\epsilon) \abs{\nabla u}) \, dx. 
\end{align*}
The first term on the right-hand side goes to zero by Lemma~\ref{lemma:remove}, and the second term tends to zero by the Lebesgue dominated convergence theorem.
\end{proof}

Recall that $u_+$ and $u_{-}$ are defined by \eqref{def_upm} when $\alpha>1$.
\begin{corollary}\label{corr:ge1}
Let $\alpha >1$ and $u\in W^{1, \Phi(\cdot)}(\Omega)$. If $u_+ = u_{-}$ then $u \in H^{1, \Phi(\cdot)}(\Omega)$. If $u \in W_0^{1, \Phi(\cdot)}(\Omega)$ and $u_+ = u_{-}$ then $u\in H_0^{1, \Phi(\cdot)}(\Omega)$.
\end{corollary}
\begin{proof}
Set $u(0) = u_{+} = u_{-}$. Due to Lemma~\ref{lem:embedding} we have
\begin{gather*}
 \abs{u(x)-u(0)}\lesssim  \log^\frac{1-\alpha}{2} \frac{1}{\abs{x}},\quad\text{if }  x\in C_{+}\cup C_{-}.%, \\
%\gamma(r):= \omega(r) \log ^\frac{\alpha-1}{2} \frac{1}{r} \to 0 \quad \text{as} \quad r\to 0.
\end{gather*}
Let 
$$
u_\epsilon = u(0) + (u-u(0)) \eta_\epsilon.
$$
By Lemma~\ref{lemma:post} we have $u_\epsilon \in  H^{1,\Phi(\cdot)}(\Omega)$. Without loss we can assume $u$ to be bounded, $\|u\|_\infty \leq M <\infty$. We claim that $u_\epsilon\to u$ in $W^{1,\Phi(\cdot)}(\Omega)$, therefore $u \in H^{1,\Phi(\cdot)}(\Omega)$. We have to check that
$$
\int_\Omega \Phi (x,\abs{ \nabla (u_\epsilon - u)})\, dx \to 0 \quad \text{as}\quad \epsilon \to 0.
$$
First we prove that 
$$
I_\epsilon:=\int_{C_{+} \cup C_{-}} \Phi(x, \abs{(u-u(0)) \nabla \eta_\epsilon} )\, dx \to 0\quad \text{ as }\quad \epsilon \to 0.
$$
 Using polar coordinates we estimate  
\begin{align*}
 I_\epsilon &\lesssim \int_{C_{+}\cup C_{-}} (u-u(0))^2\abs{\nabla \eta_\epsilon}^2\log^\alpha(e+\abs{u-u(0)}\abs{\nabla \eta_\epsilon})\, dx\\
 &\lesssim \int_{e^{-1/\epsilon}}^\epsilon \frac{\log^{1-\alpha} \frac{1}{r}}{r^2 \log^2 \frac{1}{r} \log^2 \frac{1}{\epsilon}}  \log^\alpha \left(e + \frac{2M}{r \log \frac{1}{r} \log \frac{1}{\epsilon}} \right) r\,dr  \\
&\lesssim \frac{1}{\log^2 \frac{1}{\epsilon}} \int_{e^{-1/\epsilon}}^\epsilon \frac{1}{r \log \frac{1}{r}} \, dr \lesssim \frac{1}{ \log \frac{1}{\epsilon}} \to 0  \,\text{ as } \epsilon \to 0. 
\end{align*}
For~$\beta\ge -1$ 
\begin{align*}
 J_{\epsilon}=\int_{\Omega\setminus (C_+\cup C_-)} \Phi(x,\abs{(u-u(0))}\nabla \eta_\epsilon)\, dx \to 0\quad \text{ as }\epsilon \to 0
\end{align*}
by Lemma~\ref{lemma:remove}.   

If~$\beta<-1$ then~$C(\bar \Omega)$ with the same argument as in Lemma~\ref{lem:embedding} with modulus of continuity~$\log^{\frac{1+\beta}{2}} \frac 1 t$ and convergence of~$J_\epsilon \to 0$ is by the same argument as above for~$I_\epsilon$, where~$\alpha$ is replaced by~$-\beta$. Then
\begin{align*}
 \int_\Omega \Phi (x,\abs{ \nabla (u_\epsilon - u)})\, dx  \lesssim I_\epsilon+J_\epsilon+ \int_\Omega \Phi(x,(1-\eta_\epsilon)\abs{\nabla u} )\, dx,
\end{align*} where the last term goes to zero by Lebesgue theorem. 
  
 For the second statement of the lemma regarding functions with zero boundary values we take $\eta \in C_0^\infty(\Omega)$ such that $\eta=1$ in a neighborhood of the origin and set 
 $$
 u_\epsilon = u(0) \eta + (u-u(0) \eta) \eta_\epsilon.
 $$
By  Lemma~\ref{lemma:post} we have  $u_\epsilon \in  H_0^{1,\Phi(\cdot)}(\Omega)$.  Clearly,
$$
 \nabla u_\epsilon -\nabla u   = (\nabla u - u(0) \nabla \eta) (\eta_\epsilon-1)  + (u-u(0)\eta) \nabla \eta_\epsilon.  % \eta _\epsilon \nabla u + u(0) (1- \eta_\epsilon)  \nabla \eta  + (u-u(0)) \nabla \eta_\epsilon.
$$
The first term on the right-hand side converges to zero in $L^{\Phi(\cdot)}(\Omega)$ by the Lebesgue theorem, and the second term term converges to zero by the same argument as above. Therefore, $u_\epsilon \to u$ in $H_0^{1,\Phi(\cdot)}(\Omega)$ and $u\in H_0^{1,\Phi(\cdot)}(\Omega)$.
\end{proof}

Now we prove the reverse statement.
\begin{proposition}\label{prop:coincide}
Let $\alpha>1$. If $u\in H^{1,\Phi(\cdot)}(\Omega)$ then $u_{+}= u_{-}$.
\end{proposition}
\begin{proof}

Let $u_\epsilon \to u$ in $W^{1,\Phi(\cdot)}(\Omega)$, $u_\epsilon \in C^\infty(\Omega)$.  By Lemma~\ref{lem:embedding}, $u_\epsilon$ are uniformly continuous in $\overline{C_+} \cup \overline{C_-}$ and uniformly converge to $u$ on this set. So the limit function $u$ is continuous in $\overline{C_+} \cup \overline{C_-}$. Hence its limit values $u_{+}$ and $u_{-}$ coincide.
\end{proof}

Now we are ready to prove Theorem~\ref{thm:density}.

\begin{proof}[Proof of Theorem~\ref{thm:density}]

If $-\beta \geq \alpha$ then $\Phi(x,t) \sim t^2 \log^{-\beta} (e+t)$, and is regular by standard theory (mollifications). Further we assume that $-\beta < \alpha$.

If $\alpha\leq 1$ the statement is  by Corollary~\ref{corr:le1}. 

If $\alpha>1$ and $\beta \in [0,1]$ we first use Lemma~\ref{lem:lim_value_coin} and then Corollary~\ref{corr:ge1}. 

If $\alpha>1$ and $\beta>1$ we conclude the proof by application of Lemma~\ref{lem:Lavr} and~\cite[Theorems 26, 28]{BalDieSur20}. We reproduce this argument for convenience  of the reader.  First, 
\begin{equation}\label{solenoid}
\int_\Omega b_2 \cdot \nabla \varphi\, dx =0 
\end{equation}
for any $\varphi\in C_0^\infty(\Omega)$ (recall that $b_2=\nabla^\perp v$), and hence for $\varphi \in H_0^{1,\Phi(\cdot)}(\Omega)$. Second, let  $\eta \in C_0^\infty(\Omega)$ such that $\eta=1$ in a neighborhood of the origin.  We have 
$$
\int_\Omega b_2 \cdot \nabla (\eta u_2)\, dx = \int_\Omega b_2 \cdot \nabla u_2 \, dx - \int_\Omega b_2 \cdot \nabla ((1-\eta)u_2) \, dx    =-1
$$
by Proposition~\ref{pro:prop-uAb-d}, Lemma~\ref{lemma:post} and \eqref{solenoid}. Therefore $\eta u_2 \in W_0^{1,\Phi(\cdot)}(\Omega)$ but it can not be approximated by smooth functions. The functional $w \mapsto \int_\Omega b_2 \cdot \nabla w\, dx$ is a \textit{separating functional} --- it is a nontrivial linear bounded functional on $W_0^{1,\Phi(\cdot)}(\Omega)$, vanishing on its subspace $H_0^{1,\Phi(\cdot)}(\Omega)$.  To demonstrate the Lavrentiev gap, evaluate 
$$
\mathcal{G}(u) = \int_\Omega \Phi(x, \abs{\nabla u})\, dx + \int_\Omega b_2 \cdot \nabla u\, dx
$$
on $t u_2$: by Proposition~\ref{pro:prop-uAb-d} and $\nabla_2$ property of $\Phi$ for sufficiently small $t>0$ there holds
$$
\mathcal{G}(tu_2) = \int_\Omega \Phi(x, t\abs{\nabla u})\, dx - t \leq \frac{t}{2} - t <0.
$$
 On the other hand, \eqref{solenoid} implies $\mathcal{G}(w) \geq 0$ for any $w\in  H_0^{1,\Phi(\cdot)}(\Omega) $.

Another way to show that $\eta u_2 \notin H_0^{1,\Phi(\cdot)}(\Omega)$ is by Proposition~\ref{prop:coincide} since obviously $(\eta u_2)_{+} \neq (\eta u_2)_{-}$.

To prove that the codimension of $H_0^{1,\Phi(\cdot)}(\Omega)$ in $W_0^{1,\Phi(\cdot)}(\Omega)$ is one, we note that for any $u\in  W_0^{1,\Phi(\cdot)}(\Omega)$ the function $w = u - (u_+ - u_{-})\eta u_2$ has the same limit values $w_{+}$ and $w_{-}$:
$$
w_{+} - w_{-} = u_+ - u_{-} - (u_+-u_{-}) ((u_{2})_{+} - (u_2)_-  ) =  u_+ - u_{-} - (u_+-u_{-})  = 0.
$$
 Therefore $w\in H_0^{1,\Phi(\cdot)}(\Omega) $ by Corollary~\ref{corr:ge1}. The proof of Theorem~\ref{thm:density} is complete.
\end{proof}

Now we turn to the question of regularity of the  solution of  variational problems~$\mathcal{E}_1$ and $\mathcal{E}_2$, see \eqref{eq:E12}.  Let us start by introducing spaces with boundary values: for
$g \in H^{1,\Phix}(\Omega)$ we define
\begin{align*}
  H_g^{1,\Phix}(\Omega) &:= g + H_0^{1,\Phix}(\Omega).
\end{align*}
For $g \in W^{1,\Phix}(\Omega)$ we define
\begin{align*}
  W_g^{1,\Phix}(\Omega) &:= g + W_0^{1,\Phix}(\Omega).
\end{align*}
We can define
\begin{align*}
  h_W(g) &= \argmin \mathcal{F}\big(W^{1,\Phix}_{g}(\Omega)\big).
  % \\
  % h_HW &= \argmin \mathcal{F}\big(H^{1,\Phix}_g(\Omega)\big).
\end{align*}
Formally, it satisfies the Euler-Lagrange equation (in the weak sense)
\begin{align*}
  -\Delta_\Phix h_W := -\divergence \bigg( \frac{\Phi'(x,\abs{\nabla h_W})}{\abs{\nabla
  h_W}} \nabla h_W \bigg) &= 0 \qquad \text{ in $
                            (W^{1,\Phix}_0(\Omega))^*$},
\end{align*}
where $\Phi'(x,t)$ is the derivative with respect to~$t$.
However, we can define also
\begin{align*}
  h_H(g) &= \argmin \mathcal{F}\big(H^{1,\Phix}_g(\Omega)\big).
\end{align*}
Then

\begin{align*}
  -\Delta_\Phix h_H :=-\divergence \bigg( \frac{\Phi'(x,\abs{\nabla h_H})}{\abs{\nabla
  h_H}} \nabla h_H \bigg) &= 0 \qquad \text{ in }  (H^{1,\Phix}_0(\Omega))^*. \end{align*}
 Thus $h_W$ and $h_H$ are both~$\Phix$-harmonic but $h_W$ is
  $\Phix$-harmonic in the sense of~$W^{1,\Phix}$ and $h_H$ is
  $\Phix$-harmonic with respect to~$H^{1,\Phix}$. These solutions are different in the situation of Lavrentiev gap (see~\cite[Theorem 30]{BalDieSur20}). Indeed, let~$\eta\in C_0^\infty(\bar{\Omega})$ be zero in the neighbourhood of the origin and be 1 in the neighbourhood of~$\partial\Omega$. Then there exist sufficiently large~$t$ such that for~$g=t\eta u_2 \in H^{1,\Phi(\cdot)}(\Omega)$ we have~$h_H(g)\neq h_W(g)$.
\begin{theorem}
Let $\alpha,\beta>1$. Any $H$-minimizer $h_H$ is continuous in $\Omega$. Any $W$-minimizer $h_W$ that is not equal to $h_H$ is discontinuous at the origin. 
\end{theorem}

\begin{proof}
For $H$-minimizer $h_H$ there holds $(h_H)_{+}$ = $(h_H)_-$ by Proposition~\ref{prop:coincide}. This and Lemma~\ref{lem:embedding} give the continuity of $h_H$ on $E_R=\{x_1 = |x_2|\}\cap \Omega$ and $E_L=\{-x_1 = |x_2|\}\cap \Omega$. In $\Omega_R=\{x_1 > |x_2|\} \cap \Omega$ and $\Omega_L=\{-x_1 > |x_2|\}\cap \Omega$ the minimizer $h_H$ is a solution of $\Phix$-Laplace equation with $\Phix (x,t) =  t^2 \log^{-\beta} (e+t)$ independent of $x$. This and continuous boundary data on $E_R$ and $E_L$ guarantee, by result due to Lieberman in~\cite{Lie91}, that $u$ is continuous in $\Omega_R$ and $\Omega_L$. Therefore $u$ is continuous in $\Omega$. If $h_W \neq h_H$, then $(h_W)_+ \neq (h_W)_{-}$.
 
\end{proof}

\section{Generalized double-phase type integrands}
We consider integrands  of the type
\begin{align}
\label{eq:integrand}
 \Phi(x,t):=\phi(t)+a(x)\psi(t),
\end{align}
where~$\phi$, $\psi$ are the Orlicz functions, which satisfy~$\Delta_2$ and~$\nabla_2$ conditions and the non-standard growth conditions:
\begin{align*}
 \abs{t}^{p_-}&\le \phi(t)\le c_1\abs{t}^{p_+}+c_2, \\
 \abs{t}^{p_-}&\le \psi(t)\le c_1\abs{t}^{p_+}+c_2,  
\end{align*}
where~$1<p_-<p_+<\infty$, $c_1>0$, and the weight $a(\cdot)$ is as in the previous section, that is 
$$ 
a(x_1,x_2) = \begin{cases}
1, \quad & |x_1|<|x_2|,\\
0,\quad & |x_1| > |x_2|.
\end{cases}
$$
We also assume that~$\phi\le c_3\psi +c_4$ and 
$$
\frac{\varphi(t)}{\psi(t)}\to 0 \quad \text{as}\quad t\to \infty.
$$

We recall some well-know relations for~$N$-functions. We refer to the book~\cite{KraRut61} 
for the notion and the basic properties of~$N$-functions. A real function~$\Psi:\setR^{\ge 0}\to \setR^{\ge 0}$ is called~$N$-functions if~$\Psi(0)=0$ and there exists the derivative~$\Psi'$ of~$\Psi$. The function~$\Psi'$ is right continuous, non-decreasing and satisfies~$\Psi'(0)=0$,~$\Psi'(t)>0$ for~$t>0$ and~$\lim_{t\to \infty}\Psi'(t)=\infty$.  The function~$\Psi$ satisfies~$\Delta_2$-condition, if there there exists~$c_1>0$ such that for all~$t\ge 0$ holds~$\Psi(t)\le \Psi(2t)$. By~$(\Psi')^{-1}:\setR^{\ge 0}\to \setR^{\ge 0}$ we denote the function
\begin{align*}
 (\Psi')^{-1}(t):=\sup\set{s\in\setR^{\ge 0}:\Psi'(s)\le t}.
\end{align*}
If~$\Psi'$ is strictly  increasing then~$(\Psi')^{-1}$ is the inverse of~$\Psi'$. The function
\begin{align*}
 \Psi^*(t):=\int_0^t (\Psi')^{-1}(s)\, dx
\end{align*}
is also an~$N$-function with
\begin{align*}
 (\Psi^*)'(t)=(\Psi')^{-1}(t),\quad\text{ for }t>0.
\end{align*}
There also holds
\begin{align*}
 \Psi^*(t)=\sup_{s\ge 0}(st-\Psi(s)),\quad (\Psi^*)^*=\Psi.
\end{align*}
The following variant of Young's inequality holds
\begin{align}
\label{eq:Young}
 ts\le \delta \Psi(t)+c_\delta\Psi^*(s).
\end{align}
The classical Young's inequality corresponds to the case~$\delta=1$ and~$c_\delta=1$. For all~$t\ge 0$ we have
\begin{align}
\label{eq:orl1}
 \Psi\bigg(\frac{\Psi^*(t)}{t}\bigg)\le\Psi^*(t)\le  \Psi\bigg(\frac{2\Psi^*(t)}{t}\bigg). 
\end{align}
Uniformly in~$t\ge 0$
\begin{align*}
 \Psi(t)\eqsim \Psi'(t)t,\quad \Psi^*(\Psi'(t))\eqsim \Psi(t)
\end{align*}
By~$L^\Psi$ and~$W^{1,\Psi}$ we denote classical Orlicz and Sobolev-Orlicz spaces:
\begin{align*}
 f\in L^\Psi \quad&\text{ iff }\int\Psi(\abs{f})\, dx<\infty,\\
 f\in W^{1,\Psi}\quad&\text{iff }f,\,\nabla f\in L^\Psi.
\end{align*}
The space $L^\Psi$ is normed with
\begin{align*}
 \norm{f}_{L^\Psi(\Omega)}:=\inf\set{\lambda>0:\, \int_\Omega \Psi\bigg(\frac{\abs{f}}{\lambda}\bigg)\, dx\le 1}
\end{align*}
and is complete, separable and reflexive space if both~$\Psi$ and~$\Psi^*$ satisfy~$\Delta_2$-condition.  

As above, $\Omega = (-1,1)^2$. 

We  obtain necessary and sufficient conditions for~$H^{1,\Phi(\cdot)}(\Omega)=W^{1,\Phi(\cdot)}(\Omega)$ in terms of $\phi$ and $\psi$ for the checkerboard configuration. Thus we give a complete description of the double-phase checkerboard structure with general Orlicz functions.

\begin{theorem}[Lavrentiev gap]  \label{thm3}
Let
 \begin{equation}\label{thm3_cond1}
   \int_0\phi\bigg(\frac 1 r\bigg)\, rdr<\infty
 \end{equation}
 and 
\begin{equation}\label{thm3_cond2}
   \int_0\psi^*\bigg(\frac 1 r\bigg)\, rdr<\infty.
 \end{equation}
 Then $H^{1,\Phi}(\Omega)\neq W^{1,\Phi}(\Omega)$ and~$H_0^{1,\Phi}(\Omega)\neq W_0^{1,\Phi}(\Omega)$.

 Moreover there exists~$b \in L^{\Phi^*(\cdot)}(\Omega)$ such that
  for $\mathcal{G}(\cdot)$ defined by
 $$
 \mathcal{G}(u) =\int_\Omega \Phi(x,\abs{\nabla u})\, dx+\int_\Omega b\cdot\nabla u \, dx
 $$
  there is Lavrentiev gap
 $$
 \inf \mathcal{G} (W_0^{1,\Phi(\cdot)} (\Omega)) <  \inf \mathcal{G} (H_0^{1,\Phi(\cdot)}(\Omega)). 
 $$
 \end{theorem}
 \begin{proof}
 A direct check shows that \eqref{thm3_cond1} implies $\nabla u_2 \in L^{\Phi(\cdot)}(\Omega)$ and \eqref{thm3_cond2} implies $b_2 \in  L^{\Phi^*(\cdot)}(\Omega)$. The proof is concluded by application of~\cite[Theorems 26, 28]{BalDieSur20}.
 \end{proof}

If either of the conditions of Theorem~\ref{thm3} is violated, then there is no Lavrentiev gap.

\begin{theorem}[No gap]
\label{thm:density1}
The equality $H_0^{1,\Phi(\cdot)}=W_0^{1,\Phi(\cdot)}$ is valid in the following cases:
\begin{enumerate}
 \item \label{case1} 
 \begin{equation}\label{thm2_cond1}
\int_{0}\psi^*\bigg(\frac 1 r\bigg) r\, dr=+\infty.
 \end{equation}

 \item\label{case2}  
 \begin{equation}\label{thm2_cond2}
 \int_{0}\phi\bigg(\frac 1 r\bigg) r\, dr =+\infty. % \int_{0}\psi^*\bigg(\frac 1 r\bigg) r\, dr< \infty, \quad
 \end{equation}
\end{enumerate}
\end{theorem}

We start with a lemma that is a generalization of Lemma~\ref{lemma:post}.

\begin{lemma}\label{lemma:post1}
Let $u \in W^{1,\Phi(\cdot)}(\Omega)$ and $u=0$ in a neighbourhood of the origin. Then $u\in H^{1,\Phi(\cdot)}(\Omega)$. If $u\in W_0^{1,\Phi(\cdot)}(\Omega)$ and $u=0$ in a neighbourhood of the origin then $u\in H_0^{1,\Phi(\cdot)}(\Omega)$.
\end{lemma}
The statement is verbatim repetition, as well as the proof, which uses only the structure of the weight $a$.

The next result is a generalization of Lemma~\ref{lemma:remove}.
\begin{lemma}\label{lemma:remove1}
For  the integrand~$\Phi(x,t)$  defined in~\eqref{eq:integrand} and~$d=2$ let 
\begin{equation}\label{gen_cond1}
\int_{0}\psi^*\bigg(\frac 1 r\bigg) r\, dr=+\infty.
\end{equation}
 Then there exists a sequence of functions $\eta_\epsilon\in C^\infty (\bar{\Omega})$, $\epsilon\to 0$, such that $\eta_\epsilon =1$ outside $B_\epsilon(0)$, $\eta_\epsilon =0$ in a neighbourhood of the origin, $0\leq \eta_\epsilon \leq 1$, and $$\int_\Omega \Phi(x,\abs{\nabla \eta_\epsilon})\, dx \to 0 \text{  as }\epsilon\to 0.$$
\end{lemma}

\begin{proof}
Set
 \begin{align}
 \label{eq:eta}
  \eta_{r_1,r_2}(r):=\begin{cases}
            0,&\quad \text{ if }r<r_1,\\
            \int_{r_1}^{r}(\psi^*)'\bigg(\frac {c_{r_1,r_2}}{ \rho} \bigg)\, d\rho,&\quad \text{ if }r\in(r_1,r_2),\\
            1,&\quad \text{ if }r>r_2,
           \end{cases}
 \end{align}
where the constant $c_{r_1,r_2}$ comes from
\begin{align}
\label{eq:intpsi}
 \int_{r_1}^{r_2}(\psi^*)'\bigg(\frac {c_{r_1,r_2}} {\rho}\bigg)\,d\rho=1.
\end{align}
The structure of  the function~$\eta_{r_1,r_2}$ comes from the problem:
 \begin{align*}
  \int_{r_1<r<r_2}& \psi(\abs{\nabla \eta})\, dx\to \min,\\
  \eta|_{r_1}&=0,\\
  \eta|_{r_2}&=1.
 \end{align*}
The corresponding Euler-Lagrange equation has the form
\begin{align*}
\divergence\big(\psi'(\abs{\nabla \eta})\cdot\frac{\nabla \eta}{\abs{\nabla \eta}}\big) = 0, 
\end{align*}
so~$\eta$ should be~$\psi$-harmonic.  From the radial symmetry, $\eta=\eta(r)$, $\eta'\geq 0$,  so 
$$
\left( r\psi' (\eta')\right)' =0. 
$$
Thus 
$$
\psi' (\eta') = \frac{c}{r}, \quad c =c_{r_1,r_2}=\mathrm{const},
$$
and 
$$
\eta' (r)= (\psi')^{-1} (c/r) = (\psi^*)' (c/r).
$$
To satisfy the boundary conditions we have to require that
\begin{align*}
 \int_{r_1}^{r_2} (\psi^*)'\bigg(\frac c \rho\bigg)\, d\rho=1.
\end{align*}
So the required function $\eta=\eta_{r_1,r_2}$ is given by~\eqref{eq:eta} and~\eqref{eq:intpsi}.

It remains to show  that~$\nabla \eta_{r_1,r_2}\to 0$ in~$L^{\psi(\cdot)}(\Omega)$ for some  $r_1,r_2\to 0$.

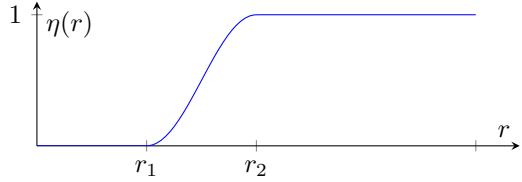
\begin{figure}[!ht]
  \centering
  \begin{tikzpicture}[
    declare function={
      func(\x)= (\x<=1/4) * (0) +
      and(\x>1/4, \x<1/2) * (1/2+1/2*sin(deg((4*\x+3/6)*3.1415))) +
      and(\x>1/2,  \x<=1) * (1);
              }
              ]
              \begin{axis}[width=8cm,height=3.5cm,
      axis x line=middle, axis y line=middle,
      ymin=0, ymax=1.1, ytick={0,1}, ylabel=$\eta(r)$,
      xmin=0, xmax=1.1, xtick={0,0.25,0.5,1}, xticklabels={$0$,$r_1$,$r_2$},xlabel=$r$,
      ]
      \addplot[blue, domain=0:1, samples=100]{func(x)};
      \draw[blue](0,0)--(0.5,0);
    \end{axis}
  \end{tikzpicture}
  \caption{The function~$\eta$}
\end{figure}%
We have
\begin{align*}
 \int_\Omega \psi(\abs{\nabla \eta_{r_1,r_2}})\, dx &= \int_{r_1}^{r_2} \psi\left((\psi^*)' \left(\frac{c_{r_1,r_2}}{\rho}  \right) \right)  \,\rho d\rho    \lesssim\int_{r_1}^{r_2}\psi^*\bigg( \frac{c_{r_1,r_2}}{\rho}\bigg)\, \rho d\rho\\
 &\lesssim \int_{r_1}^{r_2}(\psi^*)'\bigg( \frac{c_{r_1,r_2}}{\rho}\bigg) \frac{c_{r_1,r_2}}{\rho} \rho\, d\rho =c_{r_1,r_2}  \int_{r_1}^{r_2} (\psi^*)'\bigg( \frac{c_{r_1,r_2}}{\rho}\bigg)\,d\rho,
\end{align*}
where~$c_{r_1,r_2}$ is defined from~\eqref{eq:intpsi}.
Thus, 
$$
 \int_\Omega \psi(\abs{\nabla \eta_{r_1,r_2}})\, dx \leq c_{r_1,r_2}.
$$
Note that 
\begin{align}
\label{eq:div}
\int_0  \psi^*\bigg( \frac{1}{\rho}\bigg)\,\rho d\rho =\infty  \iff \int_0  (\psi^*)'\bigg( \frac{1}{\rho}\bigg)\, d\rho =\infty
\end{align}
and for any~$r_1, r_2$ the function~$\int_{r_1}^{r_2}  (\psi^*)'\bigg( \frac{c}{\rho}\bigg)\, d\rho$ is increasing with respect to~$c$.  So, if we find~$\delta$ such that
\begin{align*}
 \int_{r_1}^{r_2}  (\psi^*)'\bigg( \frac{\delta}{\rho}\bigg)\, d\rho\ge 1 ,
\end{align*}
then~$c_{r_1,r_2}\le \delta$.

For arbitrary small~$r_2$ and~$\delta$, due to~\eqref{eq:div}, we can find~$r_1<r_2$,~$r_1=r_1(r_2,\delta)$ such that
\begin{align*}
\int_{r_1}^{r_2}  (\psi^*)'\bigg( \frac{\delta}{\rho}\bigg)\, d\rho =\delta \int_{\frac{r_1}{\delta}}^{\frac{r_2}{\delta}}  (\psi^*)'\bigg( \frac{1}{\rho}\bigg)\, d\rho\ge 1. 
\end{align*}
So~$c_{r_1,r_2}\le \delta$ and thus $\int_0^1 \psi(\abs{\nabla \eta_{r_1,r_2}})\,dx \leq \delta$. The proof of the lemma is concluded by taking $\eta_\epsilon = \eta_{r_1(\epsilon,\epsilon), \epsilon}$.
\end{proof}

Recall that $\Phi^* (x,\cdot)$ is the conjugate functions of $\Phi(x,\cdot)$.
\begin{lemma}[On duality]
\label{lem:dual}
If $W^{1,\Phi^*(\cdot)}(\Omega) = H^{1, \Phi^*(\cdot)}(\Omega)$ then $W_0^{1,\Phi(\cdot)}(\Omega) = H_0^{1, \Phi(\cdot)}(\Omega)$.
\end{lemma}
\begin{proof}

 Suppose the contrary.  Then there exists~$u\in W_0^{1,\Phi^*}(\Omega)$ such that~$\nabla u$ does not belong to the closure of smooth compactly supported vector-valued functions in the gradient norm in~$(L^{\Phi(\cdot)})^2$ . Hence, there exists~$g$ such that~$g\in (L^{\Phi^*(\cdot)})^2$ and
 \begin{align*}
  &\int_\Omega g\nabla \phi\, dx=0 \quad\text{for all }\phi\in C_0^\infty(\Omega),\\
 &\int_\Omega g\nabla u\, dx\neq 0.
 \end{align*}
We see that the vector~$g$ is solenoidal. Therefore
\begin{align*}
 g=\nabla^\perp v,\quad v\in W^{1,\Phi^*(\cdot)}(\Omega).
\end{align*}
By the assumption of the lemma, there exists a sequence~$v_\epsilon\in C^\infty(\Omega)$ such that~$\nabla^\perp v_\epsilon\to g$ in~$(L^{\Phi^*(\cdot)})^2$. Since
\begin{align*}
 0=\int_\Omega \nabla^\perp v_\epsilon\nabla u\, dx\to \int_\Omega g\nabla u\, dx \neq 0,
\end{align*}
we arrive at a contradiction, which concludes the proof.
\end{proof}

The next statement is the counterpart of Corollary~\ref{corr:le1}.
\begin{corollary}\label{corr:remove1}
If 
$$
\int_{0}\psi^*\bigg(\frac 1 r\bigg) r\, dr=+\infty,
$$
then $H_0^{1,\Phi(\cdot)}(\Omega)=W_0^{1,\Phi(\cdot)}(\Omega)$ and $H^{1,\Phi(\cdot)}(\Omega)=W^{1,\Phi(\cdot)}(\Omega)$.
\end{corollary}
\begin{proof}
Let $\eta_\epsilon$ be the functions from Lemma~\ref{lemma:remove1}. Without loss of generality we can assume that~$u\in W_0^{1,\Phi(\cdot)} (\Omega)\cap L^\infty (\Omega)$. We set 
\begin{align*}
 u_\epsilon:=\eta_\epsilon u.
\end{align*}
By Lemma~\ref{lemma:post1} we have $u_\epsilon \in H_0^{1,\Phi(\cdot)}(\Omega)$. Now,
\begin{align*}
 \nabla (u_\epsilon-u)=u\nabla \eta_\epsilon+ (1-\eta_\epsilon) \nabla u.
\end{align*}
The second term converges to zero in $L^{\Phi(\cdot)}(\Omega)$ by the Lebesgue dominated convergence theorem. The first term is bounded by $\|u\|_\infty |\nabla \eta_\epsilon|$ and by the $\triangle_2$ condition and Lemma~\ref{lemma:remove1} we obtain
\begin{align*}
\int_\Omega \Phi (x, |u \nabla \eta_\epsilon|)\, dx &\lesssim \int_\Omega \psi (|u \nabla \eta_\epsilon|)\, dx  \\
&\lesssim \int_\Omega \psi (|\nabla \eta_\epsilon|)\, dx \to 0
\end{align*}
as $\epsilon \to 0$. Hence $\nabla u_\epsilon\to \nabla u $ in $L^{\Phi(\cdot)}(\Omega)$ and so $u_\epsilon \to u$ in $H_0^{1,\Phi(\cdot)}(\Omega)$. For $u\in W^{1,\Phi(\cdot)} (\Omega)\cap L^\infty (\Omega)$ the proof is the same.
\end{proof}

With these preparations in mind the proof of Theorem~\ref{thm:density1} is straightforward.

\begin{proof}[Proof of Theorem~\ref{thm:density1}]

Case~\ref{case1} is by Corollary~\ref{corr:remove1}.

Case~\ref{case2} follows by duality Lemma~\ref{lem:dual}. Indeed, $\Phi^*(x,t)  = \varphi^* (t)$ if $a(x)=0$ and $\Phi^*(x,t) \sim \psi^* (t)$ if $a(x)=1$. Clearly, $\psi^*(t) \lesssim\varphi^*(t)$ and~$\Phi^*(x,t)\lesssim \phi^*(t)$.  
  
For~$\Phi^*(x,t)$ the function $\varphi^*$ plays the role of $\psi$ and $\psi^*$ plays the role of $\varphi$. By  \eqref{thm2_cond2},
$$
\int_0 (\varphi^*)^* \bigg(\frac 1 r\bigg) r\, dr =\int_0 \varphi \bigg(\frac 1 r\bigg) r\, dr  = +\infty.
$$ 
By Corollary~\ref{corr:remove1} applied to~$\Phi^*(x,t)$ we have  $H^{1,\Phi^*(\cdot)}(\Omega) = W^{1,\Phi^*(\cdot)}(\Omega)$. The proof is concluded by Lemma~\ref{lem:dual}.
\end{proof}

%\bibliographystyle{plain}
%\bibliography{lavrentiev}

\begin{thebibliography}{10}

\bibitem{BalDieSur20}
Anna~Kh. Balci, Lars Diening, and Mikhail Surnachev.
\newblock New {E}xamples on {L}avrentiev {G}ap {U}sing {F}ractals.
\newblock {\em Calc. Var. Partial Differential Equations}, 59(5):180, 2020.

\bibitem{BarColMin15}
P.~Baroni, M.~Colombo, and G.~Mingione.
\newblock Nonautonomous functionals, borderline cases and related function
  classes.
\newblock {\em Algebra i Analiz}, 27(3):6--50, 2015.

\bibitem{BarColMin18}
Paolo Baroni, Maria Colombo, and Giuseppe Mingione.
\newblock Regularity for general functionals with double phase.
\newblock {\em Calc. Var. Partial Differential Equations}, 57(2):Art. 62, 48,
  2018.

\bibitem{ByuOh17}
Sun-Sig Byun and Jehan Oh.
\newblock Global gradient estimates for the borderline case of double phase
  problems with {BMO} coefficients in nonsmooth domains.
\newblock {\em J. Differential Equations}, 263(2):1643--1693, 2017.

\bibitem{ColMin15}
Maria Colombo and Giuseppe Mingione.
\newblock Bounded minimisers of double phase variational integrals.
\newblock {\em Arch. Ration. Mech. Anal.}, 218(1):219--273, 2015.

\bibitem{CruFio13}
David~V. Cruz-Uribe and Alberto Fiorenza.
\newblock {\em Variable {L}ebesgue spaces}.
\newblock Applied and Numerical Harmonic Analysis. Birkh\"{a}user/Springer,
  Heidelberg, 2013.
\newblock Foundations and harmonic analysis.

\bibitem{DieHHR11}
L.~Diening, P.~Harjulehto, P.~H{\"a}st{\"o}, and M.~R{\r u}{\v z}i{\v c}ka.
\newblock {\em Lebesgue and Sobolev Spaces with Variable Exponents}, volume
  2017 of {\em Lecture Notes in Mathematics}.
\newblock Springer, 1st edition, 2011.

\bibitem{EspLeoMin04}
Luca Esposito, Francesco Leonetti, and Giuseppe Mingione.
\newblock Sharp regularity for functionals with {$(p,q)$} growth.
\newblock {\em J. Differential Equations}, 204(1):5--55, 2004.

\bibitem{HarHas19}
Petteri Harjulehto and Peter H{\"a}st{\"o}.
\newblock {\em Generalized Orlicz Spaces}.
\newblock Springer International Publishing, Cham, 2019.

\bibitem{IwaSbo00}
Tadeusz Iwaniec and Carlo Sbordone.
\newblock Quasiharmonic fields.
\newblock {\em Ann. Inst. H. Poincar\'{e} Anal. Non Lin\'{e}aire},
  18(5):519--572, 2001.

\bibitem{KokMesRafSam16}
Vakhtang Kokilashvili, Alexander Meskhi, Humberto Rafeiro, and Stefan Samko.
\newblock {\em Integral operators in non-standard function spaces. {V}ol. 1},
  volume 248 of {\em Operator Theory: Advances and Applications}.
\newblock Birkh\"{a}user/Springer, [Cham], 2016.
\newblock Variable exponent Lebesgue and amalgam spaces.

\bibitem{KraRut61}
M.~A. Krasnoselskii and Ja.~B. Rutickii.
\newblock {\em Convex functions and {O}rlicz spaces}.
\newblock Translated from the first Russian edition by Leo F. Boron. P.
  Noordhoff Ltd., Groningen, 1961.

\bibitem{Lie91}
Gary~M. Lieberman.
\newblock The natural generalization of the natural conditions of
  {L}adyzhenskaya and {U}raltseva for elliptic equations.
\newblock {\em Comm. Partial Differential Equations}, 16(2-3):311--361, 1991.

\bibitem{skr20}
Igor Skrypnik and Mykhailo Voitovych.
\newblock $\mathcal{B}_{1}$ classes of ${D}$e
  ${G}$iorgi-${L}$adyzhenskaya-${U}$ral'tseva and their applications to
  elliptic and parabolic equations with generalized orlicz growth conditions.
\newblock {\em arXiv}, 2006.05764, 2020.

\bibitem{Sur14}
M.~D. Surnachev.
\newblock Density of smooth functions in weighted sobolev spaces with variable
  exponent.
\newblock {\em Doklady Mathematics}, 89(2):146--150, Mar 2014.

\bibitem{ZhiSur16}
V~V.~Zhikov and Mikhail Surnachev.
\newblock On density of smooth functions in weighted sobolev spaces with
  variable exponents.
\newblock {\em St Petersburg Mathematical Journal}, 27:415--436, 06 2016.

\bibitem{Zhi86}
V.~V. Zhikov.
\newblock Averaging of functionals of the calculus of variations and elasticity
  theory.
\newblock {\em Izv. Akad. Nauk SSSR Ser. Mat.}, 50(4):675--710, 877, 1986.

\bibitem{Zhi95}
Vasilii~V. Zhikov.
\newblock On {L}avrentiev's phenomenon.
\newblock {\em Russian J. Math. Phys.}, 3(2):249--269, 1995.

\end{thebibliography}
\end{document}